\documentclass[11pt,a4paper,english]{article}

\pdfoutput=1

\usepackage[latin1]{inputenc}
\usepackage[sort,numbers]{natbib}
\usepackage{babel}
\usepackage[centertags]{amsmath}
\usepackage{enumerate}
\usepackage{amsfonts}
\usepackage{amssymb}
\usepackage{mathtools}
\usepackage{amsthm}
\usepackage{newlfont}
\usepackage{dsfont}
\usepackage{geometry}
\usepackage{mathrsfs}
\usepackage{authblk}
\usepackage{verbatim}
\usepackage{graphicx}
\usepackage[hidelinks]{hyperref}

\DeclareMathOperator*{\essinf}{ess\,inf}
\DeclareMathOperator*{\esssup}{ess\,sup}

\newtheorem{theorem}{Theorem}[section]
\newtheorem{lemma}[theorem]{Lemma}

\newtheorem{proposition}[theorem]{Proposition}
\newtheorem{definition}[theorem]{Definition}
\newtheorem{remark}[theorem]{Remark}
\newtheorem{assumption}[theorem]{Assumption}

\newcommand{\msc}[1]{\textbf{MSC2010 Classification:} #1.}

\newcommand{\keywords}[1]{\textbf{Key words:} #1.}

\makeatletter  % so @ will be a letter
\@addtoreset{equation}{section}

\makeatother  % so  @ will not be a letter
\begin{document}
\title{\textbf{
Solving finite time horizon Dynkin games by optimal switching}\footnote{This research was partially supported by EPSRC grant EP/K00557X/1.}}

\author{Randall Martyr\footnote{School of Mathematics, University of Manchester, Oxford Road, Manchester M13 9PL, United Kingdom. email: \texttt{randall.martyr@gmail.com}}}
\maketitle
\begin{abstract}
This paper uses recent results on continuous-time finite-horizon optimal switching problems with negative switching costs to prove the existence of a saddle point in an optimal stopping (Dynkin) game. Sufficient conditions for the game's value to be continuous with respect to the time horizon are obtained using recent results on norm estimates for doubly reflected backward stochastic differential equations. This theory is then demonstrated numerically for the special cases of cancellable call and put options in a Black-Scholes market.
\end{abstract}
%%%%%%%%%%%%%%%%%%%%%%%%%%%%%%%%%%%%%%%%%%%%%%%%%%%%%%%%%%%%%%%%%%%%%%%%%%%%%%%%%%%%%%%%%%%%%
\msc{91A55, 91A05, 93E20, 60G40, 91B99, 62P20, 91A15, 91G60}
\vspace{+4pt}

\noindent\keywords{optimal switching, stopping times, optimal stopping problems, Snell envelope} 
%%%%%%%%%%%%%%%%%%%%%%%%%%%%%%%%%%%%%%%%%%%%%%%%%%%%%%%%%%%%%%%%%%%%%%%%%%%%%%%%%%%%%%%%%%%%%
\section{Introduction}
Recent papers such as \cite{Guo2008b,Hamadene2007,Yushkevich2002} have shown a connection between Dynkin games and optimal switching problems with two modes. In particular, letting $0 < T < \infty$ denote the horizon, the results of \cite{Guo2008b,Hamadene2007} show that the value process $(V_{t})_{0 \le t \le T}$ of a Dynkin game in continuous time (Section~\ref{Section:DynkinGamesIntro} below) exists and satisfies $V_{t} = Y^{1}_{t} - Y^{0}_{t}$, where $Y^{1} = (Y^{1}_{t})_{0 \le t \le T}$ and $Y^{0} = (Y^{0}_{t})_{0 \le t \le T}$ are the respective value processes for the optimal switching problem with initial mode $1$ and $0$. Separately, the papers \cite{Dumitrescu2014,Kobylanski2014} have shown how to construct two non-negative supermartingales that solve a Dynkin game on a finite time horizon. Furthermore, appropriate debut times of these supermartingales can be used to form a saddle point strategy for the game.

It is therefore apparent that \emph{classical two-player} Dynkin games and two-mode optimal switching problems are strongly coupled in the following sense: starting with either the Dynkin game or optimal switching problem, one can use its parameters and solution to formulate and solve the other problem. This paper complements these findings by proving, under appropriate conditions, that the solution to a two-mode optimal switching problem furnishes the existence of a saddle point for the corresponding Dynkin game. This is accomplished by the method of Snell envelopes which appears in \cite{Djehiche2009} for optimal switching problems on one hand, and in \cite{Dumitrescu2014,Kobylanski2014} for Dynkin games on the other hand. In the process, we relate the solution pair to the two-mode optimal switching problem to a pair of supermartingales which lie between the early exit values of the game. This condition is referred to in some contexts as \emph{Mokobodski's hypothesis}.

The content of this paper is as follows. Section~\ref{Section:Preliminaries} introduces the Dynkin game and its auxiliary optimal switching problem. Section~\ref{Section:PreliminaryAssumptions} then outlines some notation and standing assumptions. The main result on the existence of equilibria in the Dynkin game is presented in Section~\ref{Section:NashEquilibrium}. Additional results on the dependence of the game's solution on the time horizon are discussed in Section~\ref{Section:Continuous-Dependence}. Numerics which showcase this theory can be found in Section~\ref{Section:Examples}, followed by the conclusion, acknowledgements and references.

\section{Preliminaries}\label{Section:Preliminaries}
\subsection{The Dynkin game}\label{Section:DynkinGamesIntro}
Optimal stopping games, also referred to as stochastic games of timing or Dynkin games, were introduced by Eugene Dynkin sometime during the 1960s. These games have been studied extensively since then and have garnered renewed interest due to the introduction of Game Contingent Claims (also known as Israeli Options) in \cite{Kifer2000}. The particular variant of the Dynkin game which is described below was studied in recent papers such as \cite{Dumitrescu2014,Guo2008b,Hamadene2006}.

We work on a given complete probability space $\left(\Omega,\mathcal{F},\mathsf{P}\right)$ which is equipped with a filtration $\mathbb{F} = (\mathcal{F}_{t})_{0 \le t \le \infty}$ satisfying $\mathcal{F} = \mathcal{F}_{\infty} \coloneqq \bigvee_{t}\mathcal{F}_{t}$ and the usual conditions of right-continuity and completeness. We use $\mathbf{1}_{A}$ to represent the indicator function of a set (event) $A$. The shorthand notation a.s. means ``almost surely''. For $0 \le T \le \infty$ set $\mathbb{F}_{T} = (\mathcal{F}_{t})_{0 \le t \le T}$, and for each $t \in [0,T]$ let $\mathcal{T}_{t,T}$ denote the set of $\mathbb{F}_{T}$-stopping times $\nu$ which satisfy $t \le \nu \le T$ \hspace{1bp} $\mathsf{P}$-a.s. For a given $S \in \mathcal{T}_{0,T}$, we write $\mathcal{T}_{S,T} = \{\nu \in \mathcal{T}_{0,T} \colon  \nu \ge S \enskip \mathsf{P}-a.s.\}$. Let $\mathsf{E}$ denote the corresponding expectation operator. For notational convenience the dependence on $\omega \in \Omega$ is often suppressed. A horizon $T \in (0,\infty)$ is fixed for the discussion which follows and for the majority of this paper. However, we often emphasise the dependence on $T$ since the horizon is varied below in Section~\ref{Section:Continuous-Dependence}.

Let $t \in [0,T]$ be given and associate with two players $MIN$ and $MAX$ the stopping times $\sigma \in \mathcal{T}_{t,T}$ and $\tau \in \mathcal{T}_{t,T}$. The game between $MIN$ and $MAX$ is played from time $t$ until $\sigma \wedge \tau$, where $x \wedge y \coloneqq \min(x,y)$. During this time $MIN$ pays $MAX$ at a (random) rate of $\psi(t)$ per unit time. If $MIN$ exits the game prior to $T$ and either before or at the same time that $MAX$ exits, $\sigma < T$ and $\sigma \le \tau$, $MIN$ pays $MAX$ the amount $\gamma_{-}(\sigma)$. Alternatively, if $MAX$ exits the game first, $\tau < \sigma$, then $MAX$ pays to $MIN$ the amount $\gamma_{+}(\tau)$. If neither player exits the game before time $T$, we set $\sigma = \tau = T$ and $MIN$ pays $MAX$ the amount $\Gamma$. We define this payoff for the Dynkin game on $[t,T]$ in terms of the conditional expected cost to player $MIN$:
\begin{align}\label{Definition:CostFunction}
D_{t,T}(\sigma,\tau) = {} & \mathsf{E}\biggl[\int_{t}^{\sigma \wedge \tau} \psi(s){d}s + \gamma_{-}(\sigma)\mathbf{1}_{\lbrace \sigma \le \tau \rbrace}\mathbf{1}_{\lbrace \sigma < T \rbrace} - \gamma_{+}(\tau)\mathbf{1}_{\lbrace \tau < \sigma \rbrace} \nonumber \\
& \qquad + \Gamma\mathbf{1}_{\lbrace \sigma = \tau = T \rbrace} \biggm\vert \mathcal{F}_{t}\biggr], \quad \sigma,\tau \in \mathcal{T}_{t,T}
\end{align}

This is a zero-sum game since costs (gains) for $MIN$ are the gains (costs) for $MAX$. For a given $t \in [0,T]$, Player $MIN$ chooses the strategy $\sigma \in \mathcal{T}_{t,T}$ to minimise $D_{t,T}(\sigma,\tau)$ whereas $MAX$ plays the strategy $\tau \in \mathcal{T}_{t,T}$ to maximise it. This leads to upper and lower values for the game on $[t,T]$, which are denoted by $V^{+}_{t}$ and $V^{-}_{t}$ respectively:
\begin{equation}\label{Definition:UpperLowerValues}
V^{+}_{t} = \essinf\limits_{\sigma \in \mathcal{T}_{t,T}}\esssup\limits_{\tau \in \mathcal{T}_{t,T}}D_{t,T}(\sigma,\tau), \qquad V^{-}_{t} = \esssup\limits_{\tau \in \mathcal{T}_{t,T}}\essinf\limits_{\sigma \in \mathcal{T}_{t,T}}D_{t,T}(\sigma,\tau)
\end{equation}

\begin{definition}[Game Value]\label{Definition:GameValue}
	The Dynkin game on $[t,T]$ is said to be ``fair'' if there is equality between the time-$t$ upper and lower values,
	\begin{equation}\label{Definition:StackelbergEquilibrium}
	\essinf\limits_{\sigma \in \mathcal{T}_{t,T}}\esssup\limits_{\tau \in \mathcal{T}_{t,T}}D_{t,T}(\sigma,\tau) = V_{t} = \esssup\limits_{\tau \in \mathcal{T}_{t,T}}\essinf\limits_{\sigma \in \mathcal{T}_{t,T}}D_{t,T}(\sigma,\tau).
	\end{equation}
\end{definition}
The common value, denoted by $V_{t}$, is also referred to as the solution or value of the game on $[t,T]$.

When studying Dynkin games, the first course of action is to verify that the game is fair. Afterwards, one searches for strategies for the players which give the game's value or approximates it closely. This leads to the concept of a \emph{Nash equilibrium}.
\begin{definition}[Nash equilibrium]
	A pair of stopping times $(\sigma^{*},\tau^{*}) \in \mathcal{T}_{t,T} \times \mathcal{T}_{t,T}$ is said to constitute a \emph{Nash equilibrium} or a \emph{saddle point} for the game on $[t,T]$ if for any $\sigma,\tau \in \mathcal{T}_{t,T}$:
	\begin{equation}\label{Definition:NashEquilibriumPoint}
	D_{t,T}(\sigma^{*},\tau) \le D_{t,T}(\sigma^{*},\tau^{*}) \le D_{t,T}(\sigma,\tau^{*})
	\end{equation}
\end{definition}
It is not difficult to verify that the existence of a saddle point $(\sigma^{*},\tau^{*}) \in \mathcal{T}_{t,T} \times \mathcal{T}_{t,T}$ implies the game on $[t,T]$ is fair and its value is given by:
\begin{equation}\label{Result:NashEquilibriumImpliesStackelberg}
\essinf\limits_{\sigma \in \mathcal{T}_{t,T}}\esssup\limits_{\tau \in \mathcal{T}_{t,T}}D_{t,T}(\sigma,\tau) = D_{t,T}(\sigma^{*},\tau^{*}) = \esssup\limits_{\tau \in \mathcal{T}_{t,T}}\essinf\limits_{\sigma \in \mathcal{T}_{t,T}}D_{t,T}(\sigma,\tau)
\end{equation}
Under quite mild integrability and regularity assumptions on $\psi$ and $\gamma_{\pm}$, it is known (for example \cite{Ekstrom2008}) that there exists a c\`{a}dl\`{a}g $\mathbb{F}_{T}$-adapted process $(V_{t})_{0 \le t \le T}$ such that for each $t$ the random variable $V_{t}$ gives the fair value of the Dynkin game on $[t,T]$. Furthermore, if the stopping costs $\gamma_{\pm}$ are sufficiently regular then the debut times $D^{+}_{t}$ and $D^{-}_{t}$ defined by
\[
D^{+}_{t} \coloneqq \inf \lbrace s \ge t \colon V_{s} = -\gamma_{+}(s) \rbrace \wedge T,\hspace{1em} D^{-}_{t} \coloneqq \inf \lbrace s \ge t \colon V_{s} = \gamma_{-}(s) \rbrace \wedge T
\]
form a saddle point $\left(D^{-}_{t},D^{+}_{t}\right)$ for the Dynkin game on $[t,T]$. We arrive at a similar conclusion in this paper using two-mode optimal switching.

\subsection{Two-mode optimal switching}\label{Section:TwoModesOptimalSwitching}
The two-mode optimal switching or ``starting and stopping'' problem has been studied in a variety of contexts as the papers \cite{Guo2008b,Hamadene2007} and the references therein can attest. Following convention, we denote the two modes by $0$ and $1$. For $i \in \{0,1\}$ there is a random profit rate $\psi_{i} \colon \Omega \times [0,T] \to \mathbb{R}$ and time $T$ reward $\Gamma_{i} \colon \Omega \to \mathbb{R}$. For each $(i,j) \in \{0,1\} \times \{0,1\}$ there is a cost for switching from $i$ to $j$ determined by the mapping $\gamma_{i,j} \colon \Omega \times [0,T] \to \mathbb{R}$.

\begin{definition}[Auxiliary two-mode switching problem parameters]\label{Definition:AuxSwitchingParameters}
	Define parameters for the optimal switching problem from the payoff \eqref{Definition:CostFunction} of the Dynkin game as follows:
	\begin{description}
		\item[Switching costs:] For $i \in \{0,1\}$, set $\gamma_{ii}(\cdot) = 0$, $\gamma_{i,1-i}(t) \coloneqq \gamma_{-}(t)\mathbf{1}_{\lbrace i = 0 \rbrace} + \gamma_{+}(t)\mathbf{1}_{\lbrace i = 1 \rbrace}$.
		\item[Profit rate:] Set $\psi_{1}(\cdot) \equiv \psi(\cdot)$ and $\psi_{0}(\cdot) \equiv 0$.
		\item[Terminal reward:] Set $\Gamma_{1} \equiv \Gamma$ and $\Gamma_{0} \equiv 0$.
	\end{description}
\end{definition}

\begin{definition}[Admissible switching controls]\label{Definition:SwitchingControl}
	For a fixed time $t \in [0,T]$ and initial mode $i \in \lbrace 0,1\rbrace$, an admissible switching control $\alpha = (\tau_{n},\iota_{n})_{n \ge 0}$ consists of:
	\begin{enumerate}
		\item a non-decreasing sequence $\lbrace \tau_{n}\rbrace_{n \ge 0} \subset \mathcal{T}_{t,T}$ with $\tau_{0} = t$ $\mathsf{P}$-a.s.
		\item a sequence $\lbrace \iota_{n} \rbrace_{n \ge 0}$, where $\iota_{0} = i$ is the fixed initial value, $\iota_{n} \colon \Omega \to \lbrace0,1 \rbrace$ is $\mathcal{F}_{\tau_{n}}$-measurable and satisfies $\iota_{2n} = i$ and $\iota_{2n+1} = 1-i$ for $n \ge 0$.
		\item The stopping times $\lbrace \tau_{n}\rbrace_{n \ge 0}$ are finite in the following sense:
		\[
		\mathsf{P}\left(\{\tau_{n} < T,\hspace{1bp} \forall n \ge 0\}\right) = 0
		\]
		\item The (double) sequence $\alpha$ satisfies
		\[
		\mathsf{E}\bigl[\sup\nolimits_{n}\left|C^{\alpha}_{n}\right|\bigr] < \infty
		\]
		where $C^{\alpha}_{n}$ is the total cost of the first $n \ge 1$ switches under $\alpha$:
		\[
		C^{\alpha}_{n} \coloneqq \sum\limits_{k = 1}^{n}\gamma_{\iota_{k-1},\iota_{k}}(\tau_{k})\mathbf{1}_{\{\tau_{k} < T\}}, \quad n \ge 1
		\]
	\end{enumerate}
	Let $\mathcal{A}_{t,i}$ denote the set of admissible switching controls. We write $\mathcal{A}_{i}$ when $t = 0$ and drop the superscript $i$ when the initial mode is not important for the discussion.
\end{definition}

Associated with each $\alpha \in \mathcal{A}$ is a (random) function $\mathbf{u} \colon \Omega \times [0,T] \to \lbrace 0,1 \rbrace$ referred to as the mode indicator function:
\begin{equation*}\label{Definition:modeIndicator}
\mathbf{u}_{t} \coloneqq \iota_{0}\mathbf{1}_{[\tau_{0},\tau_{1}]}(t) + \sum\limits_{n \ge 1}\iota_{n}\mathbf{1}_{(\tau_{n},\tau_{n+1}]}(t),\hspace{1em} t \in [0,T]
\end{equation*}

The objective function for the switching control problem associated with the Dynkin game on $[t,T]$ is given by,
\begin{equation}\label{Definition:SwitchingControlObjective}
J(\alpha;t,i) = \mathsf{E}\left[\int_{t}^{T} \psi_{\mathbf{u}_{s}}(s){d}s + \Gamma_{\mathbf{u}_{T}} - \sum_{n \ge 1} \gamma_{\iota_{n}-1,\iota_{n}}(\tau_{n})\mathbf{1}_{\lbrace \tau_{n} < T \rbrace} \biggm \vert \mathcal{F}_{t} \right], \hspace{1em} \alpha \in \mathcal{A}_{t,i}.
\end{equation}

Together with appropriate integrability assumptions on $\psi$ and $\Gamma$, the objective function is well-defined for any $\alpha \in \mathcal{A}$. For $(t,i) \in [0,T] \times \{0,1\}$ given and fixed, the goal is to find a control $\alpha^{*} \in \mathcal{A}_{t,i}$ that maximises the performance index:
\[
J(\alpha^{*};t,i) = \esssup\limits_{\alpha \in \mathcal{A}_{t,i}}J\left(\alpha;t,i\right)
\]

\begin{remark}
	Processes or functions with super(sub)-scripts in terms of the random mode indicators $\iota_{n}$ are interpreted in the following way:
	\begin{align*}
	Y^{\iota_{n}} & = \sum\limits_{j \in \{0,1\}}\mathbf{1}_{\lbrace \iota_{n} = j \rbrace} Y^{j}, \quad n\ge 0 \\
	\gamma_{\iota_{n-1},\iota_{n}}\left(\cdot\right) & = \sum\limits_{j \in \{0,1\}}\sum\limits_{k \in \{0,1\}}\mathbf{1}_{\lbrace \iota_{n-1} = j \rbrace}\mathbf{1}_{\lbrace \iota_{n} = k \rbrace} \gamma_{j,k}\left(\cdot\right), \quad n \ge 1.
	\end{align*}
\end{remark}

\section{Notation and assumptions}\label{Section:PreliminaryAssumptions}
\subsection{Notation}
In this paper we frequently refer to concepts such as ``predictable'' and ``quasi-left-continuous'' from the general theory of the stochastic processes. The reader may consult reference texts such as \cite{Jacod2003,Rogers2000b} for further details. We note that we follow the convention of \cite{Rogers2000a,Rogers2000b} for predictable times and processes (defined on the parameter set $(0,\infty)$).

\begin{enumerate}
	\item For $p \ge 1$, let $L^{p}$ denote the set of random variables $Z$ satisfying $\mathsf{E}\left[|Z|^{p}\right] < \infty$.
	\item For $p \ge 1$, let $\mathcal{M}^{p}$ denote the set of $\mathbb{F}$-progressively measurable, real-valued processes $X = \left(X_{t}\right)_{t \ge 0}$ satisfying,
	\[
	\mathsf{E}\left[\int_{0}^{\infty}|X_{t}|^{p}{d}t\right] < \infty.
	\]
	\item For $p \ge 1$, let $\mathcal{S}^{p}$ denote the set of $\mathbb{F}$-progressively measurable processes $X$ satisfying:
	\[
	\mathsf{E}\left[\left(\sup\limits_{t \ge 0}\left|X_{t}\right|\right)^{p}\right] < \infty.
	\]
	\item Let $\mathcal{Q}$ denote the set of $\mathbb{F}$-adapted, c\`{a}dl\`{a}g processes which are quasi-left-continuous (left-continuous over stopping times).
\end{enumerate}
For a given $0 < T < \infty$ we use the analogous notation $\mathcal{M}^{p}_{T}$, $\mathcal{S}^{p}_{T}$ and $\mathcal{Q}_{T}$ for the finite time horizon $[0,T]$.
\subsection{Assumptions}
In this section $T \in (0,\infty)$ is arbitrary.
\begin{assumption}\label{Assumption:SaddlePointAssumptions}
	We impose the following integrability, measurability and regularity assumptions:
	\begin{itemize}
		\item The filtration $\mathbb{F} = (\mathcal{F}_{t})_{t \ge 0}$ satisfies the usual conditions and is quasi-left-continuous;
		\item The instantaneous payoff rate satisfies $\psi \in \mathcal{M}_{T}^{2}$;
		\item The early-exit stopping costs for the game satisfy $\gamma_{-},\gamma_{+} \in \mathcal{S}_{T}^{2} \cap \mathcal{Q}_{T}$;
		\item The terminal payoff satisfies $\Gamma \in L^{2}$ and is $\mathcal{F}_{T}$-measurable.
	\end{itemize}
\end{assumption}

\begin{assumption}\label{Assumption:StoppingCostsTerminalData}
	Stopping costs assumptions:
	\begin{align}
	i. \quad & -\gamma_{+}(T) \le \Gamma \le \gamma_{-}(T) \quad \mathsf{P}-\text{a.s.}\label{eq:TerminalDataStoppingCosts}\\
	ii. \quad & \forall t \in [0,T]: \quad \gamma_{-}(t) + \gamma_{+}(t) > 0 \quad \mathsf{P}-a.s. \label{eq:NoArbitrage}
	\end{align}
\end{assumption}

Condition~\eqref{eq:TerminalDataStoppingCosts} is standard in the literature on Dynkin games \cite{Ekstrom2008} whilst condition~\eqref{eq:NoArbitrage} is typical of optimal switching problems \cite{Guo2008b}.
\section{Existence of a Nash equilibrium via optimal switching}\label{Section:NashEquilibrium}
In this section we use martingale methods to prove for every $t \in [0,T]$ that there exists a saddle point $(\sigma^{*}_{t},\tau^{*}_{t})$ for the Dynkin game on $[t,T]$ with payoff~\eqref{Definition:CostFunction}.
\subsection{The Snell envelope}
Remember that an $\mathbb{F}_{T}$-progressively measurable process $X$ is said to belong to class $[D]$ if the set of random variables $\lbrace X_{\tau}, \tau \in \mathcal{T}_{0,T} \rbrace$ is uniformly integrable.
\begin{proposition}\label{Proposition:SnellEnvelopeProperties}
	Let $G = (G_{t})_{0 \le t \le T}$ be an adapted, $\mathbb{R}$-valued, c\`{a}dl\`{a}g process that belongs to class $[D]$. Then there exists a unique (up to indistinguishability), adapted $\mathbb{R}$-valued c\`{a}dl\`{a}g process $Z = (Z_{t})_{0 \le t \le T}$ such that $Z$ is the smallest supermartingale which dominates $G$. The process $Z$ is called the Snell envelope of $G$ and it enjoys the following properties.
	\begin{enumerate}
		\item For any $\theta \in \mathcal{T}_{0,T}$ we have:
		\begin{equation}
		Z_{\theta} = \esssup_{\tau \in \mathcal{T}_{\theta,T}}\mathsf{E}\left[G_{\tau} \vert \mathcal{F}_{\theta}\right],\text{ and therefore }Z_{T} = U_{T}.
		\end{equation}
		\item Meyer decomposition: There exist a uniformly integrable c\`{a}dl\`{a}g martingale $M$ and a predictable integrable increasing process $A$ such that for all $0 \le t \le T$,
		\begin{equation}\label{SnellEnvelope:MeyerDecomposition}
		Z_{t} = M_{t} - A_{t}, \hspace{1em} A_{0} = 0.
		\end{equation}
		\item Let $\theta \in \mathcal{T}_{0,T}$ be given and $\{\tau_{n}\}_{n \ge 0} \subset \mathcal{T}_{\theta,T}$ be an increasing sequence of stopping times tending to a limit $\tau \in \mathcal{T}_{\theta,T}$ and such that $\mathsf{E}\left[G^{-}_{\tau_{n}}\right] < \infty$ for $n \ge 0$. Suppose the following condition is satisfied for any such sequence,
		\[
		\limsup_{n \to \infty} G_{\tau_{n}} \le G_{\tau}
		\]
		Then $\tau^{*}_{\theta} \in \mathcal{T}_{\theta,T}$ defined by
		\begin{equation}\label{eq:DebutTime}
		\tau^{*}_{\theta} = \inf\lbrace t \ge \theta \colon Z_{t} = G_{t} \rbrace \wedge T
		\end{equation}
		is optimal after $\theta$ in the sense that:
		\[
		Z_{\theta} = \mathsf{E}\left[Z_{\tau^{*}_{\theta}} \vert \mathcal{F}_{\theta}\right] = \mathsf{E}\left[G_{\tau^{*}_{\theta}} \vert \mathcal{F}_{\theta}\right] = \esssup_{\tau \in \mathcal{T}_{\theta,T}}\mathsf{E}\left[G_{\tau} \vert \mathcal{F}_{\theta}\right]
		\]
		\item For every $\theta \in \mathcal{T}_{0,T}$, if $\tau^{*}_{\theta}$ is the stopping time defined in equation~\eqref{eq:DebutTime}, then the stopped process $\left(Z_{t \wedge \tau^{*}_{\theta}}\right)_{\theta \le t \le T}$ is a (uniformly integrable) c\`{a}dl\`{a}g martingale.
	\end{enumerate}
\end{proposition}
Proofs for these properties can be found in \cite{ElKaroui1981,Morimoto1982,Peskir2006} for instance.
\subsection{The martingale approach to optimal switching problems}\label{Section:OptimalSwitchingSnellEnvelopes}
Under Assumptions~\ref{Assumption:SaddlePointAssumptions} and \ref{Assumption:StoppingCostsTerminalData}, we can prove that there exists a unique pair of processes $\left(Y^{0}_{t},Y^{1}_{t}\right)_{0 \le t \le T}$ such that for $i \in \lbrace 0,1 \rbrace$, $Y^{i}$ solves the optimal switching problem in a probabilistic sense. This can be accomplished using the theory of Snell envelopes and the details can be found in a separate paper \cite{Martyr2014b}.

\begin{theorem}\label{Theorem:OptimalSwitchingProcesses}
	There exists a unique pair of processes $\left(Y^{0}_{t},Y^{1}_{t}\right)_{0 \le t \le T}$ belonging to $\mathcal{S}_{T}^{2} \cap \mathcal{Q}_{T}$ satisfying $\mathsf{P}-a.s.$,
	\begin{equation}\label{eq:OptimalSwitchingProcesses}
	\begin{cases}
	Y^{i}_{t} = \esssup\limits_{\theta \in \mathcal{T}_{t,T}} \mathsf{E}\left[\int_{t}^{\theta} \psi_{i}(s){d}s + \Gamma_{i}\mathbf{1}_{\lbrace \theta = T \rbrace} + \left\lbrace Y^{1-i}_{\theta} - \gamma_{i,1-i}(\theta) \right\rbrace \mathbf{1}_{\lbrace \theta < T\rbrace} \biggm \vert \mathcal{F}_{t}\right] \\
	Y^{i}_{T} = \Gamma_{i}
	\end{cases}
	\end{equation}
	where $i \in \lbrace 0,1 \rbrace$ and $0 \le t \le T$. Furthermore, for every $(t,i) \in [0,T] \times \lbrace 0,1 \rbrace$, there exists a control $\alpha^{*} \in \mathcal{A}_{t,i}$ such that
	\begin{equation*}\label{eq:DynamicOptimalSwitchingProblem}
	Y^{i}_{t} = J(\alpha^{*};t,i) = \esssup\limits_{\alpha \in \mathcal{A}_{t,i}}J\left(\alpha;t,i\right)
	\end{equation*}
\end{theorem}
\subsection{Existence of a Nash Equilibrium}
Let $Y^{0}$ and $Y^{1}$ be the processes in Theorem~\ref{Theorem:OptimalSwitchingProcesses} and define $G^{i} = \left(G^{i}_{t}\right)_{0 \le t \le T}$, $i \in \{0,1\}$, by:
\begin{equation}\label{eq:GainProcess2mode}
G^{i}_{t} = \Gamma_{i}\mathbf{1}_{\lbrace t = T \rbrace} + \left\lbrace Y^{1-i}_{t} - \gamma_{i,1-i}(t) \right\rbrace \mathbf{1}_{\lbrace t < T\rbrace}
\end{equation}
The process $\left(G^{i}_{t} + \int_{0}^{t} \psi_{i}(s){d}s\right)_{0 \le t \le T}$ is c\`{a}dl\`{a}g and in $\mathcal{S}_{T}^{2}$. By Proposition~\ref{Proposition:SnellEnvelopeProperties} above, the process $\left(Y^{i}_{t} + \int_{0}^{t} \psi_{i}(s){d}s\right)_{0 \le t \le T}$ is the Snell envelope of $\left(G^{i}_{t} + \int_{0}^{t} \psi_{i}(s){d}s\right)_{0 \le t \le T}$.
By Assumptions~\ref{Assumption:SaddlePointAssumptions} and \ref{Assumption:StoppingCostsTerminalData}, and as $Y^{i} \in \mathcal{S}_{T}^{2} \cap \mathcal{Q}_{T}$ for $i \in \mathbb{I}$, $G^{i}$ is quasi-left-continuous on $[0,T)$ with a possible positive jump at $T$. We can therefore apply property 3 of Proposition~\ref{Proposition:SnellEnvelopeProperties} to verify that for any $t \in [0,T]$, the stopping time $\rho^{i,*}_{t}$ defined by
\begin{equation}\label{eq:OptimalStoppingTime}
\rho^{i,*}_{t} = \inf\lbrace s \ge t \colon Y^{i}_{s} = Y^{1-i}_{s} - \gamma_{i,1-i}(s) \rbrace \wedge T
\end{equation}
is the optimal first switching time on $[t,T]$ when starting in mode $i \in \{0,1\}$. For each $t \in [0,T]$, use \eqref{eq:OptimalStoppingTime} to define a pair of stopping times $(\sigma^{*}_{t},\tau^{*}_{t})$ by
\begin{equation}\label{eq:OptimalSwitchingTimes}
\sigma^{*}_{t} = \rho^{0,*}_{t}, \qquad \tau^{*}_{t}  = \rho^{1,*}_{t}
\end{equation}
We will prove that $(\sigma^{*}_{t},\tau^{*}_{t})$ is a saddle point for the Dynkin game on $[t,T]$. In order to do so, we first establish the following lemma which relates the pair $\left(Y^{0},Y^{1}\right)$ to \emph{Mokobodski's hypothesis}. 
\begin{lemma}\label{Lemma:MokobodskiCondition}
	The processes $Y^{0}$ and $Y^{1}$ of Theorem~\ref{Theorem:OptimalSwitchingProcesses} satisfy the following condition:
	\begin{equation}\label{eq:MokobodskiConditionPrimer}
	\forall \tau \in \mathcal{T}_{0,T}:\hspace{1em} -\gamma_{+}(\tau) \le Y^{1}_{\tau} - Y^{0}_{\tau} \le \gamma_{-}(\tau),\hspace{1em}\mathsf{P}-a.s.
	\end{equation}
\end{lemma}
\begin{proof}
	For $i \in \lbrace 0,1 \rbrace$, let $G^{i} = \left(G^{i}_{t}\right)_{0 \le t \le T}$ be defined as in equation~\eqref{eq:GainProcess2mode}. Remember that $Y^{i}_{t} + \int_{0}^{t} \psi_{i}(s){d}s$ is the Snell envelope of $G^{i}_{t} + \int_{0}^{t} \psi_{i}(s){d}s$ on $0 \le t \le T$. Let $\tau \in \mathcal{T}_{0,T}$ be arbitrary. By the dominating property of the (right-continuous) Snell envelope, $Y^{i}_{\tau} \ge G^{i}_{\tau}$ holds $\mathsf{P}$-a.s. and this shows
	\[
	0 \le  Y^{i}_{\tau} - G^{i}_{\tau} = Y^{i}_{\tau} + \gamma_{i,1-i}(\tau) - Y^{1-i}_{\tau} \enskip \text{almost surely on} \enskip \lbrace \tau < T \rbrace 
	\]
	From this we obtain
	\[
	-\gamma_{+}(\tau) \le Y^{1}_{\tau} - Y^{0}_{\tau} \le \gamma_{-}(\tau) \enskip \text{almost surely on} \enskip \lbrace \tau < T \rbrace
	\]
	On the other hand, we have $Y^{1}_{\tau} - Y^{0}_{\tau} = \Gamma\hspace{1em}\mathsf{P}-\text{a.s.}$ on the event $\lbrace \tau = T \rbrace$. Using this with equation~\eqref{eq:TerminalDataStoppingCosts} gives
	\[
	-\gamma_{+}(\tau) \le Y^{1}_{\tau} - Y^{0}_{\tau} \le \gamma_{-}(\tau) \enskip \text{almost surely on} \enskip \lbrace \tau = T \rbrace
	\]
	and the claim~\eqref{eq:MokobodskiConditionPrimer} holds. 
\end{proof}

\begin{theorem}\label{Theorem:ExistenceOfSaddlePoint}
	Let $Y^{0}$ and $Y^{1}$ be the processes in Theorem~\ref{Theorem:OptimalSwitchingProcesses}. Then for every $t \in [0,T]$, $(\sigma^{*}_{t},\tau^{*}_{t})$ defined in equation~\eqref{eq:OptimalSwitchingTimes} satisfies:
	\begin{equation}\label{Proposition:OptimalSwitchingProcessesAndDynkinGamePayoffClaim}
	Y^{1}_{t} - Y^{0}_{t} = D_{t,T}(\sigma^{*}_{t},\tau^{*}_{t}) \quad \mathsf{P}-\text{a.s.}
	\end{equation}
	where $D_{t,T}(\cdot,\cdot)$ is the payoff~\eqref{Definition:CostFunction}. Furthermore, for any $\sigma,\tau \in \mathcal{T}_{t,T}$:
	\begin{equation}\label{Proposition:OptimalSwitchingProcessesAndSaddlePointClaim}
	D_{t,T}(\sigma^{*}_{t},\tau) \le D_{t,T}(\sigma^{*}_{t},\tau^{*}_{t}) \le D_{t,T}(\sigma,\tau^{*}_{t})
	\end{equation}
\end{theorem}
\begin{proof}
	The claim is trivially satisfied for $t = T$, so henceforth let $t \in [0,T)$ be a given but arbitrary time. For $i \in \lbrace 0,1 \rbrace$, let $G^{i} = \left(G^{i}_{t}\right)_{0 \le t \le T}$ be defined as in equation~\eqref{eq:GainProcess2mode}. Define a process $\hat{Y}^{1} = \left(\hat{Y}^{1}_{t}\right)_{0 \le t \le T}$ by $\hat{Y}^{1}_{t} \coloneqq Y^{1}_{t} + \int_{0}^{t} \psi(r){d}r$. By Theorem \MakeUppercase{\romannumeral 2}.77.4 of \cite{Rogers2000a}, a stopped supermartingale is also a supermartingale. For every $\sigma,\tau \in \mathcal{T}_{t,T}$ the stopped Snell envelopes $\left(Y^{0}_{s \wedge (\sigma \wedge \tau^{*}_{t})}\right)_{t \le s \le T}$ and $\left(\hat{Y}^{1}_{s \wedge (\sigma^{*}_{t} \wedge \tau)} \right)_{t  \le s \le T}$ are therefore supermartingales. Additionally using the martingale property of the stopped Snell envelope in Proposition~\ref{Proposition:SnellEnvelopeProperties}, we see that $\hat{Y}^{1} - Y^{0}$ satisfies the following:
	\begin{enumerate}
		\item $\left(\hat{Y}^{1}_{s} - Y^{0}_{s}\right)_{t \le s \le (\sigma^{*}_{t} \wedge \tau^{*}_{t})} \text{ is a martingale}$;
		\item for any $\sigma,\tau \in \mathcal{T}_{t,T}$, $\left(\hat{Y}^{1}_{s} - Y^{0}_{s}\right)_{t \le s \le (\sigma^{*}_{t} \wedge \tau)} \text{ is a supermartingale}$;
		\item for any $\sigma,\tau \in \mathcal{T}_{t,T}$, $\left(\hat{Y}^{1}_{s} - Y^{0}_{s}\right)_{t \le s \le (\sigma \wedge \tau^{*}_{t})} \text{ is a submartingale}$.
	\end{enumerate}
	This characterisation enables us to prove both \eqref{Proposition:OptimalSwitchingProcessesAndDynkinGamePayoffClaim} and \eqref{Proposition:OptimalSwitchingProcessesAndSaddlePointClaim}. The arguments used to establish \eqref{Proposition:OptimalSwitchingProcessesAndSaddlePointClaim} are essentially the same as which we use to show \eqref{Proposition:OptimalSwitchingProcessesAndDynkinGamePayoffClaim}, modulo straightforward changes from equalities to inequalities based on Assumption~\ref{Assumption:StoppingCostsTerminalData} and Lemma~\ref{Lemma:MokobodskiCondition}. We therefore only prove \eqref{Proposition:OptimalSwitchingProcessesAndDynkinGamePayoffClaim}.
	
	The martingale property of $\hat{Y}^{1} - Y^{0}$ on $[t,\sigma^{*}_{t} \wedge \tau^{*}_{t}]$ allows us to deduce the following:
	\begin{equation}\label{eq:MartingaleCharacterisation1}
	Y^{1}_{t} - Y^{0}_{t} = \mathsf{E}\left[\int_{t}^{\sigma^{*}_{t} \wedge \tau^{*}_{t}} \psi(r){d}r + Y^{1}_{\sigma^{*}_{t} \wedge \tau^{*}_{t}} - Y^{0}_{\sigma^{*}_{t} \wedge \tau^{*}_{t}} \biggm \vert \mathcal{F}_{t}\right]
	\end{equation}
	
	The term involving the pair $\left(Y^{0},Y^{1}\right)$ inside of the conditional expectation may be rewritten as:
	\begin{align}\label{eq:MartingaleCharacterisationStoppingCosts}
	\mathsf{E}\left[Y^{1}_{\sigma^{*}_{t} \wedge \tau^{*}_{t}} - Y^{0}_{\sigma^{*}_{t} \wedge \tau^{*}_{t}} \big\vert \mathcal{F}_{t} \right] = {} & \mathsf{E}\bigl[\bigl(Y^{1}_{\sigma^{*}_{t}} - Y^{0}_{\sigma^{*}_{t}}\bigr)\mathbf{1}_{\left\lbrace \sigma^{*}_{t} \le \tau^{*}_{t} \right\rbrace} \big\vert \mathcal{F}_{t} \bigr] \nonumber \\
	& \qquad + \mathsf{E}\bigl[\bigl(Y^{1}_{\tau^{*}_{t}} - Y^{0}_{\tau^{*}_{t}}\bigr)\mathbf{1}_{\left\lbrace \tau^{*}_{t} < \sigma^{*}_{t} \right\rbrace} \big\vert \mathcal{F}_{t} \bigr]
	\end{align}
	
	By equation~\eqref{eq:OptimalSwitchingTimes} and conditional on the event $\left\lbrace \tau^{*}_{t} < T \right\rbrace$, optimality of the stopping time $\tau^{*}_{t}$ gives the following:
	\begin{equation}\label{eq:MAXPLAYERSTOPSFIRST}
	Y^{1}_{\tau^{*}_{t}}\mathbf{1}_{\left\lbrace \tau^{*}_{t} < T \right\rbrace} = \left[-\gamma_{+}\left(\tau^{*}_{t}\right) + Y^{0}_{\tau^{*}_{t}}\right]\mathbf{1}_{\left\lbrace \tau^{*}_{t} < T \right\rbrace}
	\end{equation}
	
	Furthermore, $\mathbf{1}_{\lbrace \sigma^{*}_{t} > \tau^{*}_{t} \rbrace} = \mathbf{1}_{\lbrace \sigma^{*}_{t} > \tau^{*}_{t} \rbrace}\mathbf{1}_{\lbrace \tau^{*}_{t} \le T\rbrace} = \mathbf{1}_{\lbrace \sigma^{*}_{t} > \tau^{*}_{t} \rbrace}\mathbf{1}_{\lbrace \tau^{*}_{t} < T\rbrace}$ since $\tau^{*}_{t} \le T$ and $\sigma^{*}_{t} \le T$ $\mathsf{P}$-a.s., and we can use equation~\eqref{eq:MAXPLAYERSTOPSFIRST} to verify the following: $\mathsf{P}$-a.s.,
	\begin{align}\label{eq:EXPECTEDCOSTMAXPLAYERSTOPSFIRST}
	\mathsf{E}\left[\left(Y^{1}_{\tau^{*}_{t}} - Y^{0}_{\tau^{*}_{t}}\right)\mathbf{1}_{\left\lbrace \tau^{*}_{t} < \sigma^{*}_{t} \right\rbrace} \big\vert \mathcal{F}_{t}\right] & = \mathsf{E}\left[\left(Y^{1}_{\tau^{*}_{t}} - Y^{0}_{\tau^{*}_{t}}\right)\mathbf{1}_{\left\lbrace \tau^{*}_{t} < \sigma^{*}_{t} \right\rbrace}\mathbf{1}_{\lbrace \tau^{*}_{t} < T\rbrace} \big\vert \mathcal{F}_{t} \right] \nonumber \\
	& = \mathsf{E}\left[\left(-\gamma_{+}\left(\tau^{*}_{t}\right)\right)\mathbf{1}_{\left\lbrace \tau^{*}_{t} < \sigma^{*}_{t} \right\rbrace} \big\vert \mathcal{F}_{t} \right]
	\end{align}
	
	By equation~\eqref{eq:OptimalSwitchingTimes} and conditional on the event $\left\lbrace \sigma^{*}_{t} < T \right\rbrace$, optimality of the stopping time $\sigma^{*}_{t}$ gives:
	\[
	Y^{0}_{\sigma^{*}_{t}}\mathbf{1}_{\left\lbrace \sigma^{*}_{t} < T \right\rbrace} = \left[-\gamma_{-}\left(\sigma^{*}_{t}\right) + Y^{1}_{\sigma^{*}_{t}}\right]\mathbf{1}_{\left\lbrace \sigma^{*}_{t} < T \right\rbrace}
	\]
	which is used to deduce:
	\begin{equation}\label{eq:EXPECTEDCOSTMINPLAYERSTOPSFIRST}
	\mathsf{E}\left[\left(Y^{1}_{\sigma^{*}_{t}} - Y^{0}_{\sigma^{*}_{t}}\right)\mathbf{1}_{\left\lbrace \sigma^{*}_{t} \le \tau^{*}_{t} \right\rbrace}\mathbf{1}_{\left\lbrace \sigma^{*}_{t} < T \right\rbrace} \big\vert \mathcal{F}_{t} \right] = \mathsf{E}\left[\gamma_{-}\left(\sigma^{*}_{t}\right)\mathbf{1}_{\left\lbrace \sigma^{*}_{t} \le \tau^{*}_{t} \right\rbrace}\mathbf{1}_{\left\lbrace \sigma^{*}_{t} < T \right\rbrace} \big\vert \mathcal{F}_{t} \right]
	\end{equation}
	
	Since $\tau^{*}_{t} \le T$ $\mathsf{P}$-a.s. we have $\mathbf{1}_{\left\lbrace \sigma^{*}_{t} \le \tau^{*}_{t} \right\rbrace}\mathbf{1}_{\left\lbrace \sigma^{*}_{t} = T \right\rbrace} = \mathbf{1}_{\left\lbrace \sigma^{*}_{t} = \tau^{*}_{t} = T \right\rbrace}$, and using $Y^{1}_{T} = \Gamma$ and $Y^{0}_{T} = 0$ a.s., we get:
	\begin{equation}\label{eq:EXPECTEDTERMINALCOST}
	\mathsf{E}\left[\left(Y^{1}_{\sigma^{*}_{t}} - Y^{0}_{\sigma^{*}_{t}}\right)\mathbf{1}_{\left\lbrace \sigma^{*}_{t} \le \tau^{*}_{t} \right\rbrace}\mathbf{1}_{\left\lbrace \sigma^{*}_{t} = T \right\rbrace} \big\vert \mathcal{F}_{t} \right] = \mathsf{E}\left[\Gamma\mathbf{1}_{\left\lbrace \sigma^{*}_{t} = \tau^{*}_{t} = T \right\rbrace} \big\vert \mathcal{F}_{t}\right]
	\end{equation}
	
	Again, since $\sigma^{*}_{t} \le T$ $\mathsf{P}$-a.s., we can use equations~\eqref{eq:EXPECTEDCOSTMINPLAYERSTOPSFIRST} and~\eqref{eq:EXPECTEDTERMINALCOST} to assert:
	\begin{align}\label{eq:EXPECTEDCOSTMINPLAYER}
	\mathsf{E}\left[\left(Y^{1}_{\sigma^{*}_{t}} - Y^{0}_{\sigma^{*}_{t}}\right)\mathbf{1}_{\left\lbrace \sigma^{*}_{t} \le \tau^{*}_{t} \right\rbrace} \big\vert \mathcal{F}_{t} \right] = {} & \mathsf{E}\left[\left(Y^{1}_{\sigma^{*}_{t}} - Y^{0}_{\sigma^{*}_{t}}\right)\mathbf{1}_{\left\lbrace \sigma^{*}_{t} \le \tau^{*}_{t} \right\rbrace}\left(\mathbf{1}_{\left\lbrace \sigma^{*}_{t} < T \right\rbrace} + \mathbf{1}_{\left\lbrace \sigma^{*}_{t} = T \right\rbrace}\right) \big\vert \mathcal{F}_{t} \right] \nonumber \\
	= {} & \mathsf{E}\bigl[\gamma_{-}\left(\sigma^{*}_{t}\right)\mathbf{1}_{\left\lbrace \sigma^{*}_{t} \le \tau^{*}_{t} \right\rbrace}\mathbf{1}_{\left\lbrace \sigma^{*}_{t} < T \right\rbrace} \big\vert \mathcal{F}_{t}\bigr] \nonumber \\
	& \qquad + \mathsf{E}\bigl[\Gamma\mathbf{1}_{\left\lbrace \sigma^{*}_{t} = \tau^{*}_{t} = T \right\rbrace} \big\vert \mathcal{F}_{t}\bigr]
	\end{align}
	
	We then prove the claim~\eqref{Proposition:OptimalSwitchingProcessesAndDynkinGamePayoffClaim} by using equations~\eqref{eq:MartingaleCharacterisationStoppingCosts},~\eqref{eq:EXPECTEDCOSTMAXPLAYERSTOPSFIRST} and~\eqref{eq:EXPECTEDCOSTMINPLAYER} in equation~\eqref{eq:MartingaleCharacterisation1}.
\end{proof}
\begin{remark}
	The results of Theorem~\ref{Theorem:ExistenceOfSaddlePoint} were obtained in a similar fashion to several other papers in the literature which have used probabilistic approaches. For instance, \cite{Morimoto1984a} (particularly Theorem~1) which uses martingale methods for Dynkin games; \cite{Peskir2009} (particularly Theorem~2.1) which has a semi-harmonic characterisation of the value function for the Dynkin game in a Markovian setting; and \cite{Dumitrescu2014,Hamadene2006} which use the concept of doubly reflected backward stochastic differential equations.
\end{remark}
\begin{remark}
	Although we started with a Dynkin game and subsequently formulated an optimal switching problem, we could have derived these results by doing the reverse. More precisely, take any two-mode optimal switching problem (satisfying the assumptions in Section~\ref{Section:PreliminaryAssumptions}) with terminal reward data $\Gamma_{1},\Gamma_{0}$, and instantaneous profit processes $\psi_{1},\psi_{0}$. We then formulate the corresponding Dynkin game by setting $\Gamma \coloneqq \Gamma_{1} - \Gamma_{0}$, $\psi \coloneqq  \psi_{1} - \psi_{0}$ and using the switching cost function to identify the stopping costs for the game as in Definition~\ref{Definition:AuxSwitchingParameters}.
\end{remark}
\section{Dependence of the game's solution on the time horizon}\label{Section:Continuous-Dependence}
We suppose in this section and the next that there exists a standard Brownian motion $B = (B_{t})_{t \ge 0}$ defined on $(\Omega,\mathcal{F},\mathsf{P})$, and furthermore that $\mathbb{F} = (\mathcal{F}_{t})_{t \ge 0}$ is the completed natural filtration of $B$. It is well known that in this case all $\mathbb{F}$-stopping times are predictable. Therefore, all $\mathbb{F}$-adapted processes belonging to $\mathcal{Q}$ have paths which are $\mathsf{P}$-almost surely continuous.

Suppose that $\psi$ and $\gamma_{\pm}$ of Section~\ref{Section:DynkinGamesIntro} are defined on all of $[0,\infty)$ with $\psi \in \mathcal{M}^{2}$ and $\gamma_{\pm} \in \mathcal{S}^{2} \cap \mathcal{Q}$ ($\gamma_{\pm}$ still satisfying Assumption~\ref{Assumption:StoppingCostsTerminalData}). Additionally, for simplicity and ease of notation in what follows, we suppose $\psi \equiv 0$ and define two processes $L = (L_{t})_{t \ge 0}$ and $U = (U_{t})_{t \ge 0}$ by $L_{t} = -\gamma_{+}(t)$ and $U_{t} = \gamma_{-}(t)$.

For $0 < T \le \infty$ and $t \in [0,T]$, we define the following payoff for a Dynkin game: for $\sigma,\tau \in \mathcal{T}_{t,T}$,
\begin{equation}\label{Definition:CostFunction-Infinite-Horizon}
D_{t,T}(\sigma,\tau) = \mathsf{E}\left[U_{\sigma}\mathbf{1}_{\lbrace \sigma \le \tau \rbrace}\mathbf{1}_{\lbrace \sigma < T \rbrace} + L_{\tau}\mathbf{1}_{\lbrace \tau < \sigma \rbrace} + \Gamma^{T}\mathbf{1}_{\lbrace \sigma = \tau = T \rbrace} \biggm\vert \mathcal{F}_{t}\right]
\end{equation}
where $\Gamma^{T} \in L^{2}$ is $\mathcal{F}_{T}$-measurable. In the case $T = \infty$ we assume $\liminf_{t}U_{t} \le \limsup_{t}L_{t}$ and $\Gamma^{\infty}$ satisfies either $\Gamma^{\infty} \coloneqq \limsup_{t}L_{t}$ or $\Gamma^{\infty} \coloneqq \liminf_{t}U_{t}$ as appropriate.

Under appropriate conditions in both finite and infinite horizon settings, it is known (for example \cite{Ekstrom2008}, or this paper for the finite horizon case) that there is a c\`{a}dl\`{a}g $\mathbb{F}_{T}$-adapted process $V^{T}$ such that the random variable $V^{T}_{t}$ is the value of the game with payoff~\eqref{Definition:CostFunction-Infinite-Horizon}. In this section we prove that the deterministic (since $\mathcal{F}_{0}$ is trivial) mapping $T \mapsto V^{T}_{0}$ is continuous on $(0,\infty)$. This will be obtained as a straightforward consequence of recent results in \cite{Pham2013} on norm estimates for doubly reflected backward stochastic differential equations (DRBSDEs).
\subsection{Doubly reflected backward stochastic differential equations}
In order to motivate the discussion on DRBSDEs we make the following observations. By Theorem~\ref{Theorem:OptimalSwitchingProcesses}, we know that for each $T \in (0,\infty)$ given and fixed that there exist processes $Y^{0,T}$ and $Y^{1,T}$ belonging to $\mathcal{S}^{2}_{T} \cap \mathcal{Q}_{T}$ satisfying \eqref{eq:OptimalSwitchingProcesses}. Moreover, since $\psi \equiv 0$ it is also true that $Y^{0,T}$ and $Y^{1,T}$ are Snell envelopes of appropriate processes and are therefore supermartingales. Let $(M^{i,T},A^{i,T})$ denote the Meyer decomposition for $Y^{i,T}$, $i \in \{0,1\}$ (cf. \eqref{SnellEnvelope:MeyerDecomposition}). We note that both $M^{i,T}$ and $A^{i,T}$ belong to $\mathcal{S}^{2}_{T}$ since $Y^{i,T} \in \mathcal{S}^{2}_{T}$ and the filtration $\mathbb{F}_{T}$ is quasi-left-continuous. Using this decomposition, $Y^{i,T}_{T} = \Gamma_{i}$ and Brownian martingale representation for $M^{i,T}$, we have for all $t \in [0,T]$:
\begin{equation}\label{eq:DRBSDE-Solution-1}
Y^{i,T}_{t} = \Gamma^{i,T} - \int_{t}^{T}\zeta^{i,T}_{s}{d}B_{s} + A^{i,T}_{T} - A^{i,T}_{t} \quad \mathsf{P}\text{-a.s.}
\end{equation}
where $\zeta^{i,T} \in \mathcal{M}^{2}_{T}$ is predictable. Furthermore, one can also show (for example, Proposition B.11 of \cite{Kobylanski2012}) that
\begin{equation}\label{eq:DRBSDE-Flat-Off-Condition-1}
\int_{0}^{T}\left[Y^{i,T}_{t} - (Y^{1-i,T}_{t} - \gamma_{i,1-1}(t))\right]{d}A^{i,T}_{t} = 0 \quad \mathsf{P}\text{-a.s.}
\end{equation}
Recall from Theorem~\ref{Theorem:ExistenceOfSaddlePoint} that the process $V^{T} = (V^{T}_{t})_{0 \le t \le T}$ defined by $V^{T}_{t} = Y^{1,T}_{t} - Y^{0,T}_{t}$ solves the Dynkin game with payoff \eqref{Definition:CostFunction-Infinite-Horizon}. Recalling Definition~\ref{Definition:AuxSwitchingParameters}, Lemma~\ref{Lemma:MokobodskiCondition} and using \eqref{eq:DRBSDE-Solution-1}--\eqref{eq:DRBSDE-Flat-Off-Condition-1} above, we see that on $[0,T]$ the process $V^{T}$ satisfies
\begin{gather}\label{eq:DRBSDE-Solution-2}
\begin{split}
\begin{cases}
V^{T}_{t} = \Gamma^{T} - \int_{t}^{T}\zeta^{T}_{s}{d}B_{s} + K^{T}_{T} - K^{T}_{t} \\
L \le V^{T} \le U, \quad \left[V^{T}_{t} - L_{t}\right]{d}A^{1,T}_{t} = \left[U_{t} - V^{T}_{t}\right]{d}A^{0,T}_{t} = 0
\end{cases}\\
\text{where}\enskip \zeta^{T} \coloneqq  \zeta^{1,T} - \zeta^{0,T} \text{ and } K^{T} \coloneqq A^{1,T} - A^{0,T} \qquad\enskip
\end{split}
\end{gather}

We now introduce some notation and recall some results from \cite{Pham2013}. For $0 < T < \infty$ and $\mathbb{F}_{T}$-adapted c\`{a}dl\`{a}g processes $X$ and $X'$:
\begin{itemize}
	\item $\|X\|_{\mathcal{S}^{2}_{T}} \coloneqq \left(\mathsf{E}\left[(\sup_{0 \le t \le T}|X_{t}|)^{2}\right]\right)^{\frac{1}{2}}$
	\item For $0 \le t_{1} < t_{2} \le T$, $\bigvee_{t_{1}}^{t_{2}}X$ denotes the total variation of $X$ over $(t_{1},t_{2}]$
	\item $\|(X,X')\|_{\mathcal{S}^{2}_{T}} \coloneqq \left(\|X^{+}\|^{2}_{\mathcal{S}^{2}_{T}} + \|(X')^{-}\|^{2}_{\mathcal{S}^{2}_{T}}\right)^{\frac{1}{2}}$,
	where $X^{+}$ (resp. $(X')^{-}$) is the positive (resp.) negative part of $X$ (resp. $X'$).
	\item Letting $\hat{X}_{t} = \max(X_{t},X_{t^{-}})$, $\check{X}'_{t} = \min(X'_{t},X'_{t^{-}})$:
	\begin{align*}
	\|(X,X')\|^{2}_{T} \coloneqq {} & \sup_{\pi}\mathsf{E}\biggl[\biggl(\sum_{i=0}^{n-1}\bigl(\left[\mathsf{E}[\hat{X}_{\tau_{i+1}} \vert \mathcal{F}_{\tau_{i}}] - \check{X}'_{\tau_{i}} \right]^{+} \\
	& \qquad\qquad + \left[\hat{X}_{\tau_{i}} - \mathsf{E}[\check{X}'_{\tau_{i+1}}\vert \mathcal{F}_{\tau_{i}}]\right]^{+}\bigr) \biggr)^{2}\biggr] + \|(X,X')\|_{\mathcal{S}^{2}_{T}}^{2}
	\end{align*}
	where the supremum is taken over all stopping time partitions $\pi \colon 0 = \tau_{0} \le \ldots \le \tau_{n} = T$.
\end{itemize}
\begin{definition}\label{Definition:DRBSDE}
	Following \cite[p.~10]{Pham2013}, a (global) solution to the DRBSDE associated with a coefficient (or driver) $f(\omega,t,v,z) \colon \Omega \times [0,T] \times \mathbb{R} \times \mathbb{R} \to \mathbb{R}$, an $\mathcal{F}_{T}$-measurable terminal value $\Gamma^{T}$, and respective lower and upper barriers, $L$ and $U$, is a triple $(V,\zeta,K)$ of $\mathbb{F}_{T}$-progressively measurable processes satisfying
	\begin{gather}\label{eq:DRBSDE-Solution-3}
	\begin{cases}
	V_{t} = \Gamma^{T} + \int_{t}^{T}f(s,V_{s},\zeta_{s}){d}s - \int_{t}^{T}\zeta_{s}{d}B_{s} + K_{T} - K_{t} \\
	L \le V \le U, \quad \left[V_{t-} - L_{t-}\right]{d}A^{+}_{t} = \left[U_{t-} - V_{t-}\right]{d}A^{-}_{t} = 0
	\end{cases}
	\end{gather}
	where $V$ is c\`{a}dl\`{a}g, $K$ is a process of finite variation with orthogonal decomposition $K \coloneqq A^{+} - A^{-}$, and
	\[
	\|(V,\zeta,K)\|^{2}_{T} \coloneqq \mathsf{E}\left[\bigl(\sup_{0 \le t \le T}|V_{t}|\bigr)^{2} + \int_{0}^{T}|\zeta_{t}|^{2}{d}t + \bigl(\bigvee_{0}^{T}K\bigr)^{2}\right] < \infty.
	\]
\end{definition}
Recalling equation~\eqref{eq:DRBSDE-Solution-2} above and the properties of $(V^{T},\zeta^{T},K^{T})$, we see that the triple $(V^{T},\zeta^{T},K^{T})$ is a solution to the DRBSDE~\eqref{eq:DRBSDE-Solution-2} in the sense of Definition~\ref{Definition:DRBSDE}. Moreover, using Lemma~\ref{Lemma:MokobodskiCondition} (Mokobodski's hypothesis) and Theorem 3.4 of \cite{Pham2013} for instance, we also know that $(V^{T},\zeta^{T},K^{T})$ is, modulo indistinguishability, the unique solution to \eqref{eq:DRBSDE-Solution-2} in this instance.
\subsection{Dependence of solutions to DRBSDEs on the time horizon}\label{Section:DependenceOnTimeHorizonDiscussion}
Henceforth we only consider solutions to the DRBSDE~\eqref{eq:DRBSDE-Solution-3} with $f \equiv 0$. Let us fix $T \in (0,\infty)$ and let $\{T_{n}\}_{n \ge 0} \subset (0,\infty)$ be any sequence monotonically decreasing to $T$: $T_{n} \downarrow T$. We extend the unique solution $(V,\zeta,K)$ to \eqref{eq:DRBSDE-Solution-3} on $[0,T]$ to $(V^{T},\zeta^{T},K^{T})$ defined on $[0,T_{0}]$ in the following way: For each $t \in [0,T_{0}]$, 
\begin{equation}\label{eq:ExtensionsOfDRBSDEParameters}
V^{T}_{t} \coloneqq V_{t \wedge T},  \quad \zeta^{T}_{t} \coloneqq \zeta_{t \wedge T}\mathbf{1}_{\{t \le T\}}, \quad K^{T}_{t} \equiv A^{+,T}_{t} - A^{-,T}_{t} \text{ with } A^{\pm,T}_{t} \coloneqq A^{\pm}_{t \wedge T}
\end{equation}
Defining the respective lower and upper barriers $L^{T}$ and $U^{T}$ on $[0,T_{0}]$ by $L^{T}_{t} \coloneqq L_{t \wedge T}$ and $U^{T}_{t} \coloneqq U_{t \wedge T}$, it is straightforward to check that $(V^{T},\zeta^{T},K^{T})$ is the unique solution on $[0,T_{0}]$ to the DRBSDE
\begin{gather}\label{eq:DRBSDE-Extended-Solution}
\begin{cases}
V^{T}_{t} = \Gamma^{T} - \int_{t}^{T_{0}}\zeta^{T}_{s}{d}B_{s} + K^{T}_{T_{0}} - K^{T}_{t} \\
L^{T} \le V^{T} \le U^{T}, \quad \left[V^{T}_{t-} - L^{T}_{t-}\right]{d}A^{+,T}_{t} = \left[U^{T}_{t-} - V^{T}_{t-}\right]{d}A^{-,T}_{t} = 0
\end{cases}
\end{gather}
in the sense of Definition~\ref{Definition:DRBSDE} above.
\begin{assumption}\label{Assumption:ConvergenceOfTerminalValues}
	Suppose we are given a sequence $\{\Gamma^{T_{n}}\}_{n \ge 0}$ of random variables satisfying:
	\begin{itemize}
		\item Each $\Gamma^{T_{n}}$ is $\mathbb{F}_{T_{n}}$-measurable
		\item $L_{T_{n}} \le \Gamma^{T_{n}} \le U_{T_{n}}$
		\item $\Gamma^{T_{n}} \to \Gamma^{T}$ almost surely as $n \to \infty$
		\item $\sup_{n \ge 0}|\Gamma^{T_{n}}| \in L^{2}$
	\end{itemize}
\end{assumption}
Note that the last two conditions imply $\Gamma^{T_{n}} \to \Gamma^{T}$ in $L^{2}$ as $n \to \infty$. Let $(V^{T_{n}},\zeta^{T_{n}},K^{T_{n}})$ denote the unique solution on $[0,T_{n}]$ to the DRBSDE~\eqref{eq:DRBSDE-Solution-3}. We then extend these solutions to $[0,T_{0}]$ in the same way as before (see \eqref{eq:ExtensionsOfDRBSDEParameters}--\eqref{eq:DRBSDE-Extended-Solution}), with respective lower and upper barriers $L^{T_{n}}$ and $U^{T_{n}}$. We continue writing $(V^{T_{n}},\zeta^{T_{n}},K^{T_{n}})$ to denote these extensions to avoid excessive notation.

Define $\delta^{(n)}V \coloneqq (V^{T_{n}} - V^{T})$ and similarly for other cases. Theorem 3.5 of \cite{Pham2013} proves the following estimate:
\begin{align}\label{eq:DifferenceOf2RBSDEEstimate}
{} & \mathsf{E}\left[\sup_{0 \le t \le T_{0}}\left[|\delta^{(n)}V_{t}|^{2} + |\delta^{(n)}K_{t}|^{2}\right] + \int_{0}^{T_{0}}|\delta^{(n)}\zeta_{t}|^{2}{d}t\right]
\nonumber \\
\le {} & C\mathsf{E}[|\delta^{(n)}\Gamma|^{2}] + C\biggl(\mathsf{E}[|\Gamma^{T}|^{2} + |\Gamma^{T_{n}}|^{2}] + \|(L^{T_{n}},U^{T_{n}})\|_{T_{0}} \nonumber \\
& + \|(L^{T},U^{T})\|_{T_{0}}\biggr)\left(\mathsf{E}\left[\sup_{0 \le t \le T_{0}}\left[|\delta^{(n)}L_{t}|^{2} + |\delta^{(n)}
U_{t}|^{2}\right]\right]\right)^{\frac{1}{2}}
\end{align}
where $C$ is a positive constant.
\subsection{Dependence of the value of the Dynkin game on the time horizon}
We now return to the theme of this section, which is to show $T \mapsto V^{T}_{0}$ is continuous on $(0,\infty)$. For this it suffices to show that for every $T \in (0,\infty)$ and arbitrary sequence $\{T_{n}\}_{n \ge 0} \subset (0,\infty)$ satisfying $T_{n} \to T$, that $V^{T_{n}}_{0} \to V^{T}_{0}$ with $V^{T_{n}}$ (resp. $V^{T}$) denoting the unique solution to \eqref{eq:DRBSDE-Solution-3} with $f \equiv 0$ and time horizon $[0,T_{n}]$ (resp. $[0,T]$), and where convergence takes place in the usual Euclidean sense. We argue by showing $T \mapsto V^{T}_{0}$ is right-continuous and left-continuous at each point in $(0,\infty)$, noting further that it is sufficient to prove this sequential convergence for monotone sequences $\{T_{n}\}_{n \ge 0} \subset (0,\infty)$. We only show that $T \mapsto V^{T}_{0}$ is right-continuous since the other case follows by similar reasoning.
\begin{theorem}\label{Theorem:Right-Continuous-Function-Of-The-Time-Horizon}
	Let $T \in (0,\infty)$ be arbitrary and $\{T_{n}\}_{n \ge 0} \subset (0,\infty)$ be any sequence satisfying $T_{n} \downarrow T$. Let $D_{0,T}(\cdot,\cdot)$ (resp. $D_{0,T_{n}}(\cdot,\cdot)$) be the payoff \eqref{Definition:CostFunction-Infinite-Horizon} for the Dynkin game with horizon $[0,T]$ (resp. $[0,T_{n}]$). Suppose the terminal values $\Gamma^{T}$ and $\{\Gamma^{T_{n}}\}_{n \ge 0}$ in these respective payoffs satisfy Assumption~\ref{Assumption:ConvergenceOfTerminalValues}. Then, letting $V^{T}_{0}$ and $\{V^{T_{n}}_{0}\}_{n \ge 0}$ denote the values for these games (which exist by Theorem~\ref{Theorem:ExistenceOfSaddlePoint}), we have
	\begin{equation}\label{eq:ConvergenceOfGameValues}
	\lim_{n \to \infty}|V^{T_{n}}_{0} - V^{T}_{0}|^{2} = 0 
	\end{equation}
	and the map $T \mapsto V^{T}_{0}$ is therefore right-continuous on $(0,\infty)$.
\end{theorem}
\begin{proof}
	From the discussion in Section~\ref{Section:DependenceOnTimeHorizonDiscussion} above, we can assert that there exists a positive constant $C$ such that (cf. \eqref{eq:DifferenceOf2RBSDEEstimate}):
	\begin{align}\label{eq:DifferenceBetweenGameValuesEstimate-1}
	|V^{T_{n}}_{0} - V^{T}_{0}|^{2} \le {} & C\mathsf{E}[|\delta^{(n)}\Gamma|^{2}] + C\biggl(\mathsf{E}[|\Gamma^{T}|^{2} + |\Gamma^{T_{n}}|^{2}] + \|(L^{T_{n}},U^{T_{n}})\|_{T_{0}} + \|(L^{T},U^{T})\|_{T_{0}}\biggr) \nonumber \\
	& \qquad \qquad \times \left(\mathsf{E}\left[\sup_{0 \le t \le T_{0}}\left[|\delta^{(n)}L_{t}|^{2} + |\delta^{(n)}
	U_{t}|^{2}\right]\right]\right)^{\frac{1}{2}}
	\end{align}
	Note that $\mathsf{E}[|\Gamma^{T_{n}}|^{2}]$ is uniformly bounded in $n$ since $\sup_{n \ge 0}|\Gamma^{T_{n}}| \in L^{2}$ by Assumption~\ref{Assumption:ConvergenceOfTerminalValues}. Theorem 3.4 of \cite{Pham2013} verifies that the norm $\|(L,U)\|_{T_{0}}$ is finite, and it is not difficult to see that $\|(L^{T},U^{T})\|_{T_{0}} \le \|(L^{T_{n}},U^{T_{n}})\|_{T_{0}} \le \|(L,U)\|_{T_{0}}$ for every $n$. Using this in \eqref{eq:DifferenceBetweenGameValuesEstimate-1} shows that we have
	\begin{align}\label{eq:DifferenceBetweenGameValuesEstimate-2}
	|V^{T_{n}}_{0} - V^{T}_{0}|^{2} \le {} & C\mathsf{E}[|\delta^{(n)}\Gamma|^{2}] + C\biggl(\mathsf{E}[|\Gamma^{T}|^{2} + \sup_{n \ge 0}|\Gamma^{T_{n}}|^{2}] + 2\|(L,U)\|_{T_{0}}\biggr) \nonumber \\
	& \qquad \qquad \qquad \times \left(\mathsf{E}\left[\sup_{0 \le t \le T_{0}}\left[|\delta^{(n)}L_{t}|^{2} + |\delta^{(n)}
	U_{t}|^{2}\right]\right]\right)^{\frac{1}{2}}
	\end{align}
	and the right-hand side of \eqref{eq:DifferenceBetweenGameValuesEstimate-2} is finite for all $n \ge 0$. We have
	\[
	\sup_{0 \le t \le T_{0}}\left[|\delta^{(n)}L_{t}|^{2} + |\delta^{(n)}
	U_{t}|^{2}\right] = \sup_{T \le t \le T_{n}}\left[|L_{t} - L_{T}|^{2} + |U_{t} - U_{T}|^{2}\right]
	\]
	which decreases monotonically to $0$ almost surely as $n \to \infty$. By making use of the Monotone Convergence Theorem and $\lim_{n \to \infty}\mathsf{E}[|\delta^{(n)}\Gamma|^{2}] = 0$ by Assumption~\ref{Assumption:ConvergenceOfTerminalValues}, passing to the limit $n \to \infty$ in \eqref{eq:DifferenceBetweenGameValuesEstimate-2} gives
	\[
	0 \le \liminf_{n \to \infty}|V^{T_{n}}_{0} - V^{T}_{0}|^{2} \le \limsup_{n \to \infty}|V^{T_{n}}_{0} - V^{T}_{0}|^{2} \le 0
	\]
	and the claim follows.
\end{proof}
\section{Numerical examples}\label{Section:Examples}
\subsection{Cancellable call and put options}
In this section we use the same probabilistic setup as Section~\ref{Section:Continuous-Dependence} above. We assume a Black-Scholes market with constant risk-free rate of interest $r > 0$ and risky asset price process $S = \left(S_{t}\right)_{t \ge 0}$ which satisfies
\begin{equation}\label{eq:ExponentialBrownianMotion}
S_{t} = S_{0}\exp\left(\left(r - \tfrac{\rho^{2}}{2}\right)t + \rho B_{t}\right),\quad t \ge 0
\end{equation}
where $S_{0} > 0$ and $\rho > 0$ are constants. A call (resp. put) option on the underlying asset $S$ with finite expiration $T > 0$ is a contingent claim that gives the holder the right, but not the obligation, to buy (resp. sell) the asset $S$ at a predetermined strike price $K$ by time $T$. If this option is of ``American'' style, then the holder can exercise this right at any time $\tau \in [0,T]$. The payoff $G(S_{\tau})$ of the option when exercised at time $\tau \in [0,T]$ is given by:
\begin{equation}\label{eq:OptionPayOff}
G(S_{\tau}) = \begin{cases}
(S_{\tau} - K)^{+} \quad \text{for a call option} \\
(K-S_{\tau})^{+} \quad \text{for a put option}
\end{cases}
\end{equation}

A cancellable (game) version of the option grants the writer the ability to cancel it at a premature time $0 \le \sigma < T$. If the writer decides to exercise this right, then the option holder receives the payoff of the standard option plus an additional amount $\delta > 0$, which is a penalty imposed on the writer for terminating the contract early. The expected value of the cash flow from the writer to the seller at time $0$ is given by:
\begin{equation}\label{Definition:GameCallOption}
D_{0,T}(\sigma,\tau) = \mathsf{E}\bigl[e^{-r \sigma}\left(G(S_{\sigma}) + \delta\right)\mathbf{1}_{\lbrace \sigma < \tau \rbrace}\mathbf{1}_{\lbrace \sigma < T \rbrace} + e^{-r \tau}G(S_{\tau})\mathbf{1}_{\lbrace \tau \le \sigma \rbrace}\bigr]
\end{equation}
The holder of the contract would like to choose the exercise time $\tau$ to maximise the payoff. On the other hand, the writer would like to minimise this payoff by choosing the appropriate cancellation time $\sigma$. We assume that $\sigma$ and $\tau$ are chosen from the set $\mathcal{T}_{0,T}$ of stopping times.

Equation~\eqref{Definition:GameCallOption} is the payoff for a Dynkin game between the option writer and holder (albeit slightly different to~\eqref{Definition:CostFunction} above). The assumptions listed in Section~\ref{Section:PreliminaryAssumptions} can be verified for this game, and an inspection of the proof of Theorem~\ref{Theorem:ExistenceOfSaddlePoint} shows that its conclusion remains valid for the payoff~\eqref{Definition:GameCallOption}. The cancellable call/put option can therefore be valued using optimal switching.

\subsection{Approximation procedure}
Suppose we are additionally given an integer $0 < M < \infty$ and an increasing sequence of times $\{t_{m}\}_{m=0}^{M} \subset [0,T]$ satisfying $t_{0} = 0$ and $t_{M} = T$. Set $\hat{\mathbb{F}} = \{\mathcal{F}_{t_{m}}\}_{m = 0}^{M}$ and for each $t_{m}$ and $i \in \{0,1\}$, let $\hat{\mathcal{A}}^{(M)}_{t_{m},i} \subset \mathcal{A}_{t_{m},i}$ be the subclass of controls $\alpha = (\tau_{n},\iota_{n})_{\ge 0}$ where each $\tau_{n}$ takes values in $\{t_{m},\ldots,t_{M}\}$ and satisfies $\mathsf{P}\left(\{\tau_{n} < T \} \cap \{\tau_{n} = \tau_{n+1}\}\right) = 0$ for $n \ge 1$. Our discrete-time approximation to the auxiliary optimal switching problem starting in mode $i \in \{0,1\}$ at time $t_{m}$ takes a similar form as \eqref{Definition:SwitchingControlObjective} (with $\psi_{1} = \psi_{0} = 0$ for simplicity): $\alpha \in \hat{\mathcal{A}}^{(M)}_{t_{m},i}$,
\begin{equation*}\label{eq:Discretised-Optimal-Switching}
\hat{J}^{(M)}(\alpha;t_{m},i) = \mathsf{E}\bigl[\Gamma_{\iota_{N(\alpha)}} - \sum\nolimits_{n \ge 1}\gamma_{\iota_{n-1},\iota_{n}}(\tau_{n})\mathbf{1}_{\lbrace \tau_{n} < T \rbrace} \bigm\vert \mathcal{F}_{t_{m}}\bigr]
\end{equation*}
where $\iota_{N(\alpha)}$ is the last mode switched to before $T$ under the control $\alpha$. The results of \cite{Martyr2014a} show that there exist $\hat{\mathbb{F}}$-adapted sequences $\hat{Y}^{(M),i} = \{\hat{Y}^{(M),i}_{m}\}_{m = 0}^{M}$, $i \in \{0,1\}$, defined by
\begin{equation}\label{eq:Discretised-Backward-Induction}
\begin{split}
\hat{Y}^{(M),i}_{M} & = \Gamma_{i},\hspace{1 em}\text{ and for } m = M - 1,\ldots,0 : \\
\hat{Y}^{(M),i}_{m} & = \max\limits_{j \in \{0,1\}} \left\lbrace -\gamma_{i,j}(t_{m}) + \mathsf{E}\left[\hat{Y}^{(M),j}_{m+1} \bigm\vert \mathcal{F}_{t_{m}}\right] \right\rbrace
\end{split}
\end{equation}
such that $\max_{m \in \{0,\ldots,M\}}\big|\hat{Y}^{(M),i}_{m}\big| \in L^{2}$ and $\hat{Y}^{(M),i}_{m} = \esssup_{\alpha \in \hat{\mathcal{A}}^{(M)}_{t_{m},i}}\hat{J}^{(M)}(\alpha;t_{m},i)$ $\mathsf{P}$-a.s.

For each $M = 1,2,\ldots,$ define $\hat{V}^{(M)} = \{\hat{V}^{(M)}_{m}\}_{m = 0}^{M}$ by $\hat{V}^{(M)}_{m} \coloneqq \hat{Y}^{(M),1}_{m} - \hat{Y}^{(M),0}_{m}$ and recall the particular parametrization given in Definition~\ref{Definition:AuxSwitchingParameters}. Recalling Theorem~\ref{Theorem:ExistenceOfSaddlePoint}, we see that the random variable $\hat{V}^{(M)}_{m}$ can be used to approximate the value of the continuous-time Dynkin game with payoff $D_{t_{m}}(\cdot,\cdot)$ (cf. \eqref{Definition:CostFunction}). There is, however, a more efficient backward induction formula for $\hat{V}^{(M)}$. For $m = M-1,\ldots,0$ and $i \in \{0,1\}$ define events $\mathcal{C}^{i}_{m}$ and $\mathcal{D}^{i}_{m}$ as follows:
\begin{equation}\label{eq:ProofOfGameBackwardInductionFormula-Events}
\begin{split}
\mathcal{C}^{i}_{m} & \coloneqq \left\{\hat{Y}^{(M),i}_{m} =  \mathsf{E}\left[\hat{Y}^{(M),i}_{m+1} \bigm\vert \mathcal{F}_{t_{m}}\right] \right\}\\
\mathcal{D}^{i}_{m} & \coloneqq \left\{\hat{Y}^{(M),i}_{m} = -\gamma_{i,1-i}(t_{m}) + \mathsf{E}\left[\hat{Y}^{(M),1-i}_{m+1} \bigm\vert \mathcal{F}_{t_{m}}\right]\right\}
\end{split}
\end{equation}
Notice that $\mathsf{P}(\mathcal{C}^{i}_{m} \cup \mathcal{D}^{i}_{m}) = 1$ for every $i \in \{0,1\}$ and $m = M-1,\ldots,0$. It is not difficult to verify (using Assumption~\ref{Assumption:StoppingCostsTerminalData} and optimality arguments -- see \cite{Martyr2014a}) that $\mathsf{P}(\mathcal{D}^{0}_{m} \cap \mathcal{D}^{1}_{m}) = 0$ for $m = M-1,\ldots,0$ and this leads to: $\mathsf{P}-a.s.$,
\begin{align}
\hat{Y}^{(M),i}_{m}\mathbf{1}_{\mathcal{D}^{1-i}_{m}} & = \mathsf{E}\left[\hat{Y}^{(M),i}_{m+1} \bigm\vert \mathcal{F}_{t_{m}}\right]\mathbf{1}_{\mathcal{D}^{1-i}_{m}} \label{eq:ProofOfGameBackwardInductionFormula-OneStepMartingale}\\
\hat{Y}^{(M),1}_{m} - \hat{Y}^{(M),0}_{m} & = (\hat{Y}^{(M),1}_{m} - \hat{Y}^{(M),0}_{m})\left(\sum_{i=0}^{1}\mathbf{1}_{\mathcal{D}^{i}_{m} \cap \mathcal{C}^{1-i}_{m}} + \mathbf{1}_{\mathcal{C}^{0}_{m} \cap \mathcal{C}^{1}_{m}}\right) \label{eq:ProofOfGameBackwardInductionFormula-ConditionOnEvents}
\end{align}
Using $\hat{V}^{(M)}_{m} = \hat{Y}^{(M),1}_{m} - \hat{Y}^{(M),0}_{m}$, equations~\eqref{eq:ProofOfGameBackwardInductionFormula-OneStepMartingale} and \eqref{eq:ProofOfGameBackwardInductionFormula-ConditionOnEvents}, definition~\eqref{eq:ProofOfGameBackwardInductionFormula-Events} for the events $\mathcal{C}^{i}_{m}$ and $\mathcal{D}^{i}_{m}$, and the backward induction formula \eqref{eq:Discretised-Backward-Induction}, one can show that $\hat{V}^{(M)}$ satisfies: $\mathsf{P}$-a.s.,
\[
\begin{split}
\hat{V}^{(M),i}_{M} & = \Gamma,\hspace{1 em}\text{ and for } m = M - 1,\ldots,0 : \\
\hat{V}^{(M),i}_{m} & = \min\left(\gamma_{-}(t_{m}),\max\left(-\gamma_{+}(t_{m}), \mathsf{E}\left[\hat{V}^{(M),i}_{m+1} \bigm\vert \mathcal{F}_{t_{m}}\right]\right)\right)
\end{split}
\]
In order to account for exponential discounting, assuming that the rewards and costs have not already been discounted, the backward induction formula should be written as:
\begin{equation}\label{eq:Discretised-Backward-Induction-Game-Value}
\begin{split}
\hat{V}^{(M),i}_{M} & = \Gamma,\hspace{1 em}\text{ and for } m = M - 1,\ldots,0 : \\
\hat{V}^{(M),i}_{m} & = \min\left(\gamma_{-}(t_{m}),\max\left(-\gamma_{+}(t_{m}), \mathsf{E}\left[e^{-r(t_{m+1} - t_{m})} \cdot \hat{V}^{(M),i}_{m+1} \bigm\vert \mathcal{F}_{t_{m}}\right]\right)\right)
\end{split}
\end{equation}
The reader can compare the backward induction formula~\eqref{eq:Discretised-Backward-Induction-Game-Value} to the one appearing in Theorem~2.1 of \cite{Kifer2000}. In a Markovian setting, the Least-Squares Monte Carlo regression (LSMC) method (Chapter 8, Section 6 of \cite{Glasserman2003}) can be used to numerically approximate the conditional expectation in \eqref{eq:Discretised-Backward-Induction-Game-Value}.

\subsection{Numerical results for the cancellable call and put options}
We now present numerical results for the cancellable call and put options. The backward induction formula~\eqref{eq:Discretised-Backward-Induction-Game-Value} with the LSMC algorithm was used to this effect, with simple monomials of degree $2$ used to approximate the conditional expectations. For each run of the algorithm, $10000$ sample paths $\{\hat{S}_{m}\}_{m=0}^{M}$ of the geometric Brownian motion \eqref{eq:ExponentialBrownianMotion} were simulated using antithetic sampling and the relation:
\begin{equation*}
\begin{cases}
\hat{S}_{0} = S_{0} \\
\hat{S}_{m+1} = \hat{S}_{m}\exp\bigl([r-\frac{\rho^{2}}{2}]h + \rho \sqrt{h} \cdot \xi_{m+1}\bigr),\quad m = 0,\ldots,M-1
\end{cases}
\end{equation*}
where $h = \frac{T}{M}$ is the step size and $\{\xi_{m}\}_{m = 1}^{M}$ is a sequence of I.I.D. standard normal random variables. The option's value was set to the empirical average of the results from 100 runs of the algorithm.

The same model parameters were used to value the cancellable call and put options. These parameters were obtained from \cite[p.~128]{Kuhn2007} and are as follows: $r = 0.06$, $\rho = 0.4$, $K = 100$ and $\delta = 5$. We computed option values on a finite time horizon with $T = 0.5 \times 2^{q}$, $q = 0,\ldots,8$, initial spot price $S_{0} \in \{60, 140\}$, and $M = 1000$ time steps.

\subsubsection{Numerical results for the cancellable call option.}
Figure~\ref{figure:cancellable_call_solution} below shows numerical results for the option values for $S_{0} \in \{60, 140\}$. The solid line shows finite horizon option values whilst the dotted line is the perpetual option's value. The latter was calculated using the following formula obtained from \cite{Ekstrom2006}:
\[
V^{\infty}_{0} = \begin{cases}
\delta\frac{S_{0}}{K}, & \text{ if } S_{0} \in [0,K] \\
S_{0} - K + \delta, & \text{ if } S_{0} \in (K,\infty)
\end{cases}
\]

\begin{figure}[h!]
	\centering
	\includegraphics*[width=0.8\textwidth,height=0.2\textheight]{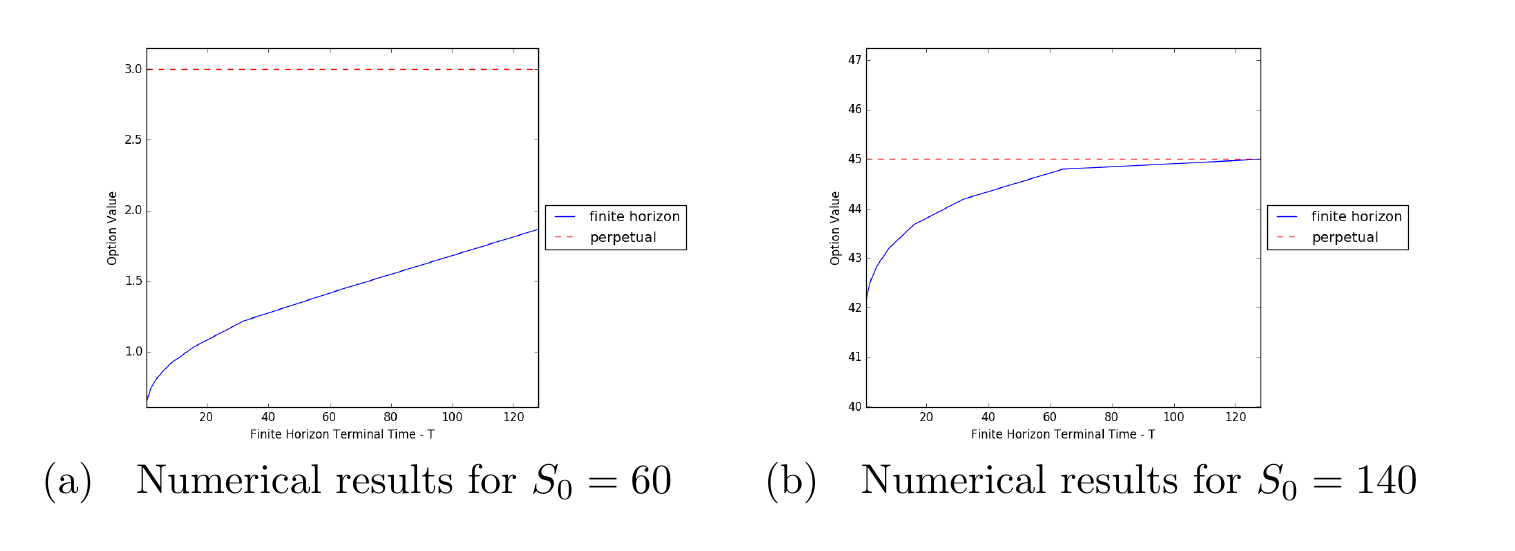}
	\caption{Finite and infinite horizon cancellable call option values for $S_{0} \in \{60, 140\}$.}
	\label{figure:cancellable_call_solution}
\end{figure}

For both cases shown in Figure~\ref{figure:cancellable_call_solution}, the finite horizon option values appear to be continuous with respect to the time horizon $T$. Furthermore, in Figure~\ref{figure:cancellable_call_solution}-(b), the option values apparently converge to the perpetual option's value as $T \to \infty$.

\subsubsection{Numerical results for the cancellable put option.}
\hspace{0em}\\
\begin{figure}[h!]
	\centering
	\includegraphics*[width=0.8\textwidth,height=0.2\textheight]{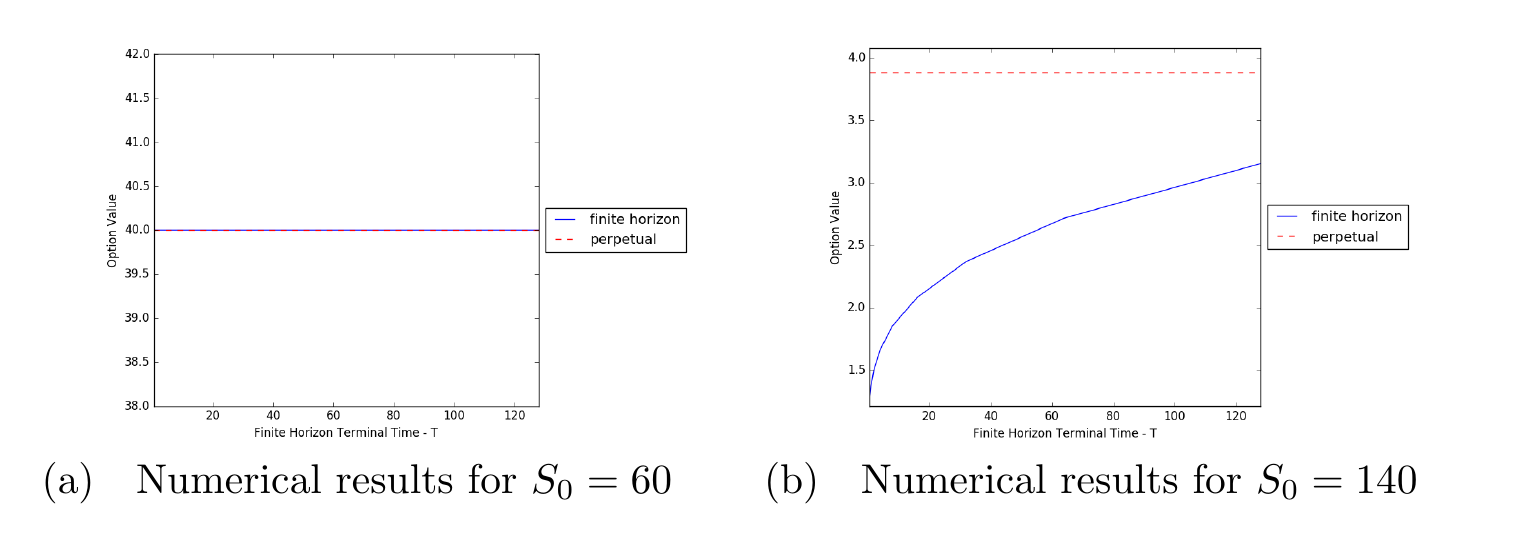}
	\caption{Finite and infinite horizon cancellable put option values for $S_{0} \in \{60, 140\}$.}
	\label{figure:cancellable_put_solution}
\end{figure}

Figure~\ref{figure:cancellable_put_solution} provides the analogous illustrations for the cancellable put option. The perpetual option's value in this case was calculated using the following formula obtained from \cite{Kyprianou2004}:
\begin{align*}
i. \quad & \delta \ge \delta^{*}: V^{\infty}_{0} = V^{AP}(S_{0}) \\
ii. \quad & \delta < \delta^{*}: V^{\infty}_{0} = \begin{cases}
K - S_{0},& \text{ if } S_{0} \in (0,k^{*}] \\
(K - k^{*})(\frac{S_{0}}{k^{*}})^{-(\gamma-1)}\frac{(\frac{S_{0}}{K})^{\gamma} - (\frac{S_{0}}{K})^{-\gamma}}{(\frac{k^{*}}{K})^{\gamma} - (\frac{k^{*}}{K})^{-\gamma}} \\
\quad +~\delta(\frac{S_{0}}{K})^{-(\gamma-1)}\frac{(\frac{S_{0}}{k^{*}})^{-\gamma} - (\frac{S_{0}}{k^{*}})^{\gamma}}{(\frac{k^{*}}{K})^{\gamma} - (\frac{k^{*}}{K})^{-\gamma}},&\text{ if } S_{0} \in (k^{*},K) \\
\delta(\frac{S_{0}}{K})^{-(2\gamma-1)},&\text{ if } S_{0} \in [K,\infty)
\end{cases}
\end{align*}
where $\gamma = \frac{r}{\rho^{2}} + \frac{1}{2}$, $S \mapsto V^{AP}(S)$ is the time $0$ value for the perpetual American put option as a function of the initial asset price, $\delta^{*} = V^{AP}(K)$, and $\frac{k^{*}}{K}$ is the solution in $(0,1)$ to the following equation:
\[
y^{2\gamma} + 2\gamma - 1 = 2\gamma\left(1 + \frac{\delta}{K}\right)y
\]
For the interested reader, we note that $\delta^{*} = V^{AP}(100) \approxeq 30.3$ and $k^{*} \approxeq 69.9$ to one decimal place. This means $V^{\infty}_{0} = K - S_{0}$ when $S_{0} = 60$ and $V^{\infty}_{0} = \delta(\frac{S_{0}}{K})^{-(2\gamma-1)}$ when $S_{0} = 140$. In terms of continuity of $T \mapsto V^{T}_{0}$ and possible convergence to the perpetual option value, from Figure~\ref{figure:cancellable_put_solution} one draws similar conclusions to those for the cancellable call option.
\section{Conclusion}
This paper showed how the solution to a two-mode optimal switching problem can be used to derive the solution to a Dynkin game in continuous-time and on a finite time horizon $[0,T]$. Under certain hypotheses, the value $V_{t}$ of the Dynkin game starting from $t \ge 0$ exists and satisfies $V_{t} = Y^{1}_{t} - Y^{0}_{t}$, where $Y^{1}_{t}$ and $Y^{0}_{t}$ are the respective optimal values for the optimal switching problem with initial mode $1$ and $0$. Furthermore, $(Y^{1}_{t})_{0 \le t \le T}$ and $(Y^{0}_{t})_{0 \le t \le T}$ (and therefore $V = (V_{t})_{0 \le t \le T}$) are right-continuous processes, and a Nash equilibrium solution to the Dynkin game can be constructed using appropriate debut times of $V$. Results on doubly reflected stochastic differential equations were used to prove that the value of the game is a continuous function of the time horizon parameter $T$. This result was confirmed via numerical experiments for cancellable call and put options.

% Appendix here
% Options are (1) APPENDIX (with or without general title) or 
%             (2) APPENDICES (if it has more than one unrelated sections)
% Outcomment the appropriate case if necessary
%
%\appendix

\section*{Acknowledgments}
% Enter the text of acknowledgments here
This research was partially supported by EPSRC grant EP/K00557X/1. The author would like to thank his PhD supervisor J. Moriarty, colleague T. De Angelis, Prof. S. Hamad\`{e}ne, and all others whose comments which led to an improved draft of the paper.

%%%%%%%%%%%%%%%%%%%%%%%%%%%%%%%%%%%%%%%%%%%%%%%%%%%%%%%%%%%%%%%%%%%%%%%%%%%%%%%%%%%%%%%%%%%%%%%%%%%%%%%%%%%%%%%%%%%%%%%%%%%%%%%%%%%%%%%%%%%%%%%%%%%%%%%%%%%%%%%%%%%%%%%%%%%%%%%%%%%%%%%%%%%%%%%%%%%%%%%%%%%%%%%%%%%%%%%%%%%%%%%%%%%%%%%%%%%%%%%%%%%%%%%%%%%%%%%%%%%%%%%%%%%%%%%%

%\bibliographystyle{hsiam} % outcomment this and next line in Case 1
%\bibliography{../../library} % if more than one, comma separated

\end{document}